\documentclass[letterpaper, 11pt,  reqno]{amsart}

\usepackage{graphicx}
\usepackage{subfigure}

\usepackage{amssymb,amsmath,amsthm,amscd}

\usepackage{amsxtra, esint, bbm, enumerate}
\usepackage{mathtools} 
\usepackage{skak, bbm} 
\usepackage{ulem}

\usepackage{hyperref}

\usepackage{dsfont}

\usepackage{mparhack}

\allowdisplaybreaks[2]

\newtheorem{theorem}{Theorem} [section]

\newtheorem{lemma}[theorem]{Lemma}
\newtheorem{proposition}[theorem]{Proposition}
\newtheorem{remark}[theorem]{Remark}

\newtheorem{definition}[theorem]{Definition}
\newtheorem{corollary}[theorem]{Corollary}

\DeclareMathOperator*{\supp}{supp}

\newcommand{\I}{\mathcal{I}}

\newcommand{\noi}{\noindent}
\newcommand{\Z}{\mathbb{Z}}
\newcommand{\R}{\mathbb{R}}
\newcommand{\C}{\mathbb{C}}
\newcommand{\T}{\mathbb{T}}

\newcommand{\TT}{\mathcal{T}}

\newcommand{\BT}{{\bf T}}

\newcommand{\M}{\mathcal{M}}

\newcommand{\NB}{\mathbb{N}}

\newcommand{\FL}{\mathcal{F}L} 
\newcommand{\FLv}{\overrightarrow{\mathcal{F}L}}

\renewcommand{\H}{\mathcal{H}}

\newcommand{\F}{\mathcal{F}}

\newcommand{\dl}{\delta}

\newcommand{\nb}{\nabla}

\newcommand{\Dl}{\Delta}
\newcommand{\eps}{\varepsilon}

\newcommand{\G}{\Gamma}
\newcommand{\ld}{\lambda}
\newcommand{\Ld}{\Lambda}
\newcommand{\s}{\sigma}
\newcommand{\Si}{\Sigma}

\newcommand{\dx}{\partial_x}
\newcommand{\dt}{\partial_t}

\newcommand{\embeds}{\hookrightarrow}

\renewcommand{\O}{\Omega}

\newcommand{\les}{\lesssim}
\newcommand{\ges}{\gtrsim}

\newcommand{\jb}[1]
{\langle #1 \rangle}

\newcommand{\W}{\mathcal{W}}

\numberwithin{equation}{section}
\numberwithin{theorem}{section}

\newcommand{\abs}[1]{\lvert#1\rvert}

\begin{document}

\title[Norm inflation with infinite loss of regularity for gIBQ]{Norm inflation with infinite loss of regularity for the generalized improved Boussinesq equation}

\author{Pierre de Roubin}

\maketitle

\begin{abstract}
In this paper, we study the ill-posedness issue for the generalized improved Boussinesq equation. In particular we prove there is norm inflation with infinite loss of regularity at general initial data in $\jb{\nb}^{-s}\big(L^2 \cap L^\infty\big)(\R)$ for any $s < 0$. This result is sharp in the $L^2$-based Sobolev scale in view of the well-posedness in $L^2(\R) \cap L^\infty(\R)$. We also show that the same result applies to the multi-dimensional generalized improved Boussinesq equation. Finally, we extend our norm inflation result to Fourier-Lebesgue, modulation and Wiener amalgam spaces.
\end{abstract}

\section{Introduction}

We consider the generalized improved Boussinesq equation (gIBq):
\begin{equation}
\begin{cases}\label{imBq}
\dt^{2}u-\dx^2 u - \dt^2 \dx^2u = \dx^2 \left(f(u) \right) \\
(u,\dt u)|_{t = 0} = (u_0,u_1), 
\end{cases}
\qquad ( t, x) \in \R \times \M, 
\end{equation}

\noi where $\M =  \R$ or $\T$ with $\T = \R / \Z$ and $f(u) = u^k$ for $k \geq 2$ an integer. The equation \eqref{imBq} appears in a wide variety of physical problems. Makhankov~\cite{Mak} derived the improved Boussinesq equation, namely with $f(u) = u^2$, in order to study ion-sound wave propagation and mentioned the case of the improved modified Boussinesq equation, with $f(u) = u^3$, as modeling nonlinear Alfv\'en waves. He also mentioned the possibility to use the linear term of \eqref{imBq} to describe waves propagating at right angles to the magnetic field. Besides, Clarkson-LeVeque-Saxton \cite{CLS} considered the cases $f(u) = u^p$, with $p = 3$ or $5$, as a model for the propagation of longitudinal deformation waves in an elastic rod. See also \cite{Tur} for further discussion on the physical aspect of \eqref{imBq}.

This equation has also attracted attention from a mathematical point of view. In particular, the well-posedness of this problem has been studied extensively and Constantin-Molinet~\cite{CM} proved local well-posedness for \eqref{imBq} in $L^2(\R) \cap L^{\infty}(\R)$ (and global well-posedness under additional conditions). See also \cite{CO, GL, Zhi}. On the other hand, it is known that \eqref{imBq} possesses some undesirable behaviour in negative Sobolev spaces. Following the work of Bourgain~\cite{Bou97}, it was showed in \cite{GL} that the solution map $\Phi \colon H^s(\R) \times H^s (\R)\to C([-T,T], H^s(\R))$ of \eqref{imBq}, for nonlinearity $f(u) = u^k$, fails to be $C^k$ for $s< 0$. This result in particular implies that we can not study the well-posedness of \eqref{imBq} via a contraction argument. Note however that their result does not imply failure of continuity for the data-to-soluton map in negative Sobolev spaces. Our main goal in this article is to complete the well-posedness issue of this problem by proving the discontinuity of the data-to-solution map of \eqref{imBq} in $\jb{\nb}^{-s}\big(L^2 \cap L^\infty\big) (\M)$ for any $s < 0$. For more clarity, let us denote $\W^{s, 2, \infty}(\M) \coloneqq W^{s, 2, \infty}(\M) \times W^{s, 2, \infty}(\M)$ with
$$
W^{s, 2, \infty}(\M) = \jb{\nb}^{-s}\big(L^2 \cap L^\infty\big) (\M) = \{ f, \jb{\nb}^sf \in L^2 (\M) \cap L^\infty (\M)\}.
$$
\noi We exhibit the following norm inflation with infinite loss of regularity behaviour:



\begin{theorem}
\label{THM:mainLinfty}
Let $k \geq 2$, $\s \in \R$ and $s < 0$. Fix $(u_0 , u_1) \in \W^{s, 2, \infty} (\M)$. Then, given any $\eps > 0$, there exists a solution $u_\eps$ to \eqref{imBq} with nonlinearity $f(u) = u^k$ on $\M$ and $t_\eps \in (0, \eps)$ such that
\begin{equation}
\label{EQ:NIwithILORLinfty}
\| (u_\eps (0), \dt u_\eps (0)) - (u_0 , u_1) \|_{ \W^{s, 2, \infty}} < \eps \quad \text{ and } \quad \| u_\eps (t_\eps) \|_{W^{\s, 2, \infty}} > \eps^{-1}.
\end{equation}
\end{theorem}


%


Note that, when $\s =s$ and $(u_0, u_1) = (0,0)$, \eqref{EQ:NIwithILORLinfty} is already called {\it norm inflation}. Since we have \eqref{EQ:NIwithILORLinfty} for any arbitrary initial data $(u_0, u_1) \in \W^{s, 2, \infty} (\M)$, we say that we have norm inflation at {\it general initial data}. Besides, Theorem \ref{THM:main} also yields \eqref{EQ:NIwithILORLinfty} for any arbitrary $\s < s$, leading to the so-called {\it infinite loss of regularity}. Observe as well that, as a corollary, Theorem~\ref{THM:mainLinfty} gives the following discontinuity of the solution map.


\begin{corollary}
\label{COR:DiscSolMap}
Let $s<0$. Then, for any $T > 0$, the solution map $\Phi \colon (u_0, u_1) \in \W^{s, 2, \infty}(\M) \to (u, \dt u) \in C([-T, T], W^{s, 2, \infty}(\M)) \times C([-T, T], W^{s, 2, \infty} (\M)) $ to the generalized improved Boussinesq equation \eqref{imBq} is discontinuous everywhere in $\W^{s, 2, \infty} (\M)$. 

\end{corollary}

Throughout the rest of this paper, we prove in fact the following result of norm inflation at general initial data with infinite loss of regularity in negative Sobolev spaces:

\begin{theorem}
\label{THM:main}
Let $k \geq 2$, $\s \in \R$ and $s < 0$. Fix $(u_0 , u_1) \in \H^s (\M) \coloneqq H^s (\M) \times H^s (\M)$. Then, given any $\eps > 0$, there exists a solution $u_\eps$ to \eqref{imBq} with nonlinearity $f(u) = u^k$ on $\M$ and $t_\eps \in (0, \eps)$ such that
\begin{equation}
\label{EQ:NIwithILOR}
\| (u_\eps (0), \dt u_\eps (0)) - (u_0 , u_1) \|_{\H^s} < \eps \quad \text{ and } \quad \| u_\eps (t_\eps) \|_{H^\s} > \eps^{-1}.
\end{equation}
\end{theorem}

Once Theorem \ref{THM:main} is proved, Theorem \ref{THM:mainLinfty} follows as a corollary. See Remark~\ref{REM:NIinWs2infty}. In fact, we first prove the usual norm inflation at general initial data, namely we prove \eqref{EQ:NIwithILOR} with $\s = s$. Then, the loss of regularity follows with a slight modification. See Remark~\ref{REM:ILOR}. Note also that, in a similar manner, Theorem \ref{THM:main} implies the discontinuity of the solution map in negative Sobolev spaces.


There are several ways to prove norm inflation results. Christ-Colliander-Tao \cite{CCT} first introduced norm inflation with a method based on a dispersionless ODE approach. See also \cite{BTz1, Xia}. On the other hand, Carles with his collaborators in \cite{AC, BC, CK} proved norm inflation with infinite loss of regularity results for nonlinear Schr\"odinger equations (NLS) by using geometric optics.


Our strategy is to follow a Fourier analytic argument that originated from the abstract work of Bejenaru-Tao \cite{BT} on quadratic nonlinear Schr\"odinger equation. Their idea was to expand the solution into a power series and to exploit the {\it high-to-low energy transfer} in one of the terms to prove discontinuity of the solution map. This approach was refined later on by  Iwabuchi-Ogawa~\cite{IO} who used it to extend the ill-posedness result in \cite{BT} into a norm inflation result. Kishimoto~\cite{Kis} then further developed these methods to prove norm inflation for the nonlinear Schr\"odinger equation. Meanwhile, Oh \cite{Oh} refined the argument of \cite{IO} in the context of cubic NLS by introducing a way to index the power series by trees, and to estimate each term separately. Forlano-Okamoto \cite{FO} proved afterwards norm inflation with an approach inspired by \cite{Oh} for nonlinear wave equations (NLW) in Sobolev spaces of negative regularities, and we use the same reasoning in our proof. For other papers with similar argument, the interested reader might turn to \cite{BH2, COW, CP, OOT, Ok, WZ}. See also \cite{Chevyrev} for an implementation of this method in probabilistic settings.




One of the key ingredients to our proof is the use of the Wiener algebra $\FL^1 (\M)$, which we define now. Given $\M = \R$ or $\T$, let $\widehat{\M}$ denote the Pontryagin dual of $\M$, i.e.
\begin{equation}
\label{DEF:PontryaginDual}
\widehat{\M} = 
\begin{cases}
\R & \text{if} \quad \M = \R, \\
\Z & \text{if} \quad \M = \T.
\end{cases}
\end{equation}

\noi Note that, when $\widehat{\M} = \Z$, we endow it with the counting measure. We can then define the following Fourier-Lebesgue spaces:

\begin{definition}[Fourier-Lebesgue spaces] \rm
\label{DEF:FLspaces}
For $s \in \R$ and $p \geq 1$, we define the Fourier-Lebesgue space $\FL^{s,p} (\M)$ as the completion of the Schwartz class of functions $\mathcal{S} ( \M)$ with respect to the norm
$$
\| f \|_{\FL^{s,p}(\M)} = \| \jb{\xi}^s \widehat{f}(\xi) \|_{L^p_\xi (\widehat{\M})},
$$

\noi where $\jb{\xi} \coloneqq (1 + \abs{\xi}^2 )^{\frac 12}$.
\end{definition}

\noi In particular, $\FL^{0,1}(\M)$ corresponds to the Wiener algebra, and its algebra property allows us to prove easily that \eqref{imBq} is analytically locally well-posed in $\FL^{0,1}(\M)$.

Another major point of our proof is the following power series expansion of a solution $u$ to \eqref{imBq} with $(u, \dt u ) |_{t=0} = \vec{u}_0$:
$$
u = \sum^\infty_{j=0} \Xi_j (\vec{u}_0 ),
$$

\noi where $\Xi_j (\vec{u}_0 )$ denotes a multilinear term in $\vec{u}_0$ of degree $(k-1)j + 1$ (for nonlinearity $f(u) = u^k$). 
More precisely, these multilinear terms are exactly the successive terms of a Picard iteration expansion. See Section \ref{PowerSeries}. Then, by explicit computation we show that $\Xi_1(\vec{u}_0)$ grows rapidly in a short time, achieving the desired growth, while we control the other terms. See Section \ref{Sec3}.
\begin{remark} \rm
In \cite{WC, WC2}, Wang-Chen studied the {\it multi-dimensional} generalized improved Boussinesq equation
\begin{equation}
\begin{cases}\label{generalizedimBq}
\dt^{2}u-\Dl u - \dt^2 \Dl u = \Dl \left(f(u) \right) \\
(u,\dt u)|_{t = 0} = (u_0,u_1), 
\end{cases}
\qquad ( t, x) \in [0, +\infty) \times \R^d, 
\end{equation}

\noi which corresponds essentially to the $d$-dimensional form of \eqref{imBq}, for $d \geq 1$. More precisely, they proved this problem is locally well-posed in $W^{2,p}\cap L^{\infty}$, for any $1 \leq p \leq \infty$, and in $H^s$ for $s \geq \frac d2$. They also showed global well-posedness under additional conditions. We claim that norm inflation with infinite loss of regularity also applies for this problem, for any dimension $d \geq 1$ and $s < 0$. See Remark \ref{RK:ProofGenIBq} for more details.
\end{remark}

\begin{remark} \rm
\label{REM:NIforFLMW}
\cite{BH} used this method to extend the result of Forlano-Okamoto~\cite{FO} and prove infinite loss of regularity for the nonlinear wave equation in some Fourier-Lebesgue, modulation and Wiener amalgam spaces\footnote{However, we point out that, while Forlano-Okamoto did not state it, their argument implies infinite loss of regularity.}. We claim that we can prove the same result for equation \eqref{imBq}. See Appendix \ref{appendixA} for more details.
\end{remark}



\section{Power series extension}
\label{PowerSeries}

In this section, we prove \eqref{imBq} with nonlinearity $f(u) = u^k$ is well-posed in the Wiener algebra, and it can be expanded into a power series. First, let us fix some notations. We define 
$$
\FLv^{s,p}(\M) \coloneqq \FL^{s,p}(\M) \times \FL^{s,p}(\M)
$$

\noi and, for more clarity, we denote $\FL^{p}(\M) \coloneqq \FL^{0, p}(\M)$ and $\FLv^{p}(\M) \coloneqq \FLv^{0,p}(\M)$.

We denote $S(t)$ the linear propagator:
\begin{equation}
\label{EQ:linearOp}
S(t)(u,v) = \cos\left(t P(D)\right) u +  \frac{\sin\left(tP(D)\right)}{P(D)} v,
\end{equation}

\noi with $P(D) \coloneqq \frac{\abs{\nb}}{\jb{\nb}}$. Namely, for any $\xi \in \M$,
$$
\F \left[P(D)f\right](\xi) = \frac{\abs{\xi}}{( 1 + \xi^2)^{1/2}} \widehat{f}(\xi) \eqqcolon \ld(\xi) \widehat{f}(\xi).
$$

\noi We also denote $\I_k$ the multilinear Duhamel operator.
\begin{equation}
\label{Eq:DuhamelOp}
\I_k(u_1, \dots, u_k)(t) = \int^{t}_0  \sin\left((t-t')P(D)\right) P(D) \prod^k_{j = 1} u_j (t') \mathrm{d}t'.
\end{equation}

\noi This gives the following Duhamel formulation for \eqref{imBq}
\begin{equation}
\label{Eq:duhamelP}
u(t) = S(t) (u_0, u_1) + \I_k (u).
\end{equation}

\noi Note that, in the aforementioned formula, we used the short-hand notation $\I_k (u) \coloneqq \I_k(u, \dots, u)$.

To index the power series we intend to create, we need the following tree structure:

\begin{definition}[$k$-ary trees]
\label{DEF:Trees}
\rm 
\begin{enumerate}
	\item Given a set $\TT$ with partial order $\leq$, we say that $b \in \TT$ with $b \leq a$ and $b \ne a$ is a child of $a \in \TT$, if  $b\leq c \leq a$ implies either $c = a$ or $c = b$. If the latter condition holds, we also say that $a$ is the parent of $b$.

	\item A tree $\TT$ is a finite partially ordered set,  satisfying the following properties\footnote{We do not identify two trees even if there is an order-preserving bijection between them.}:
\begin{itemize}
	\item Let $a_1, a_2, a_3, a_4 \in \TT$. If $a_4 \leq a_2 \leq a_1$ and  $a_4 \leq a_3 \leq a_1$, then we have $a_2\leq a_3$ or $a_3 \leq a_2$,

	\item A node $a\in \TT$ is called terminal, if it has no child. A non-terminal node $a\in \TT$, for $\TT$ a $k$-ary tree, is a node with exactly $k$ children, 

	\item There exists a maximal element $r \in \TT$, called the root node, such that $a \leq r$ for all $a \in \TT$, 

	\item $\TT$ consists of the disjoint union of $\TT^0$ and $\TT^\infty$, where $\TT^0$ and $\TT^\infty$ denote the collections of non-terminal nodes and terminal nodes, respectively.
\end{itemize}

	\item Let $\BT(j)$ denote the set of all trees of $j$-th generation, namely trees with $j$ non-terminal nodes.
\end{enumerate}
\end{definition}

Note that a tree of $j$-th generation $\TT \in \BT(j)$ has $kj + 1$ nodes. Indeed, it has $j$ non-terminal nodes by definition and an induction argument shows it has $(k-1)j + 1$ terminal nodes. Besides, we have the following bound on the number of trees of $j$-th generation:

\begin{lemma}
\label{LEM:NumberOfTrees}
There exists a constant $C_0 > 0$, depending only on $k$, such that, for any $j \in \NB$, we have
\begin{equation}
\label{EQ:NumberOfTrees}
\abs{\BT (j)} \leq \frac{C^j_0}{(1 + j)^2 } \leq C^j_0
\end{equation}
\end{lemma}

The following proof is an adaptation of the one in \cite{Oh} for ternary trees, we include it for completeness.

\begin{proof}
We prove \eqref{EQ:NumberOfTrees} by induction. Note that the right inequality is immediate, so we only have to prove
$$
\abs{\BT (j)} \leq \frac{C^j_0}{(1 + j)^2 }
$$

\noi for any $j \geq 0$. Observe first that
$$
\abs{\BT (0)} = \abs{\BT (1)} =1.
$$

\noi Then, fix $j \geq 2$. Assume equation \eqref{EQ:NumberOfTrees} holds for any $0 \leq m \leq j-1$ and take $\TT \in \BT(j)$. By Definition \ref{DEF:Trees}, there exist $k$ trees $\TT_1 \in \BT(j_1), \dots, \TT_k \in \BT(j_k )$, with $j_1 + \cdots + j_k = j -1$, such that $\TT$ is the tree consisting of a root node whose children are $\TT_1, \dots, \TT_k$. Thus, applying the induction hypothesis we get
\begin{align*}
\abs{\BT (j)} & = \sum_{\substack{j_1 + \cdots + j_k = j-1, \\ j_1, \cdots, j_k \geq 0}} \abs{\BT (j_1)} \times \cdots \times \abs{\BT (j_k)} \\
	& \leq \sum_{\substack{j_1 + \cdots + j_k = j-1, \\ j_1, \cdots,  j_k \geq 0}} \frac{C^{j_1}_0}{(1 + j_1)^2 } \times \cdots \times \frac{C^{j_k}_0}{(1 + j_k )^2 } \times \frac{(1+j)^2}{(1+j)^2} \\ 
	& \leq  k^2 \left( \sum_{j_2, \cdots j_k \geq 0} \frac{1}{(1 + j_2)^2 \times \cdots \times (1 + j_k )^2 } \right)  \frac{C^{j-1}_0}{(1 + j)^2 }
\end{align*}

\noi where we used $k \max \left( 1+ j_1 , \dots, 1 + j_k \right) \geq 1+j$ and rearranged the sum so the maximum is reached for $j_1$. Then, choosing $C_0 =  k^2 \sum_{j_2, \cdots j_k \geq 0} \frac{1}{(1 + j_2)^2 \cdots (1 + j_k )^2 } < \infty$ ends the induction.

\end{proof}

From now on, for any $\vec{\phi} \in \FLv^1$, we will associate to any tree $\TT \in \BT (j)$, $j \geq 0$, a space-time distribution $\Psi (\TT) (\vec{\phi}) \in \mathcal{D}' \left( (-T,T) \times \M \right)$ as follows:
\begin{itemize}
	\item replace a non-terminal node by Duhamel integral operator $\I_k$ with its $k$ arguments being the children of the node,
	\item replace a terminal node by $S(t) \vec{\phi}$.
\end{itemize}

\noi For any $j \geq 0$ and $\vec{\phi} \in \FLv^1$, we also define $\Xi_j$ as follows: 

\begin{equation}
\label{DEF:Xi_j}
\Xi_j (\vec{\phi}) = \sum_{\TT \in \BT (j)} \Psi(\TT)( \vec{\phi} ).
\end{equation}

\noi For instance, we have the two following terms:
$$
 \Xi_0 (\vec{\phi}) = S(t) \vec{\phi}, \quad \text{ and } \quad  \Xi_1 (\vec{\phi}) = \I_k ( S(t) \vec{\phi}, \cdots,  S(t) \vec{\phi}).
$$

\noi Let us now state some basic multilinear estimates that will be useful both for norm inflation and for local well-posedness in the Wiener algebra $\FL^1$.

\begin{lemma}
\label{LEM:TreesEstFL}
There exists $C > 0$ such that, for any $\vec{\phi} \in \FLv^1$, $j \in \NB$ and $0 < T \leq 1$, we have
\begin{equation}
\label{EQ:TreesEst1}
\big\| \Xi_j (\vec{\phi}) (T) \big\|_{\FL^1} \leq C^j T^{2j} \| \vec{\phi} \|^{(k-1)j+1}_{\FLv^1}
\end{equation}

\noi Moreover, if $j \geq 1$ and $\vec{\psi} \in \FLv^1 \cap \H^0 \left(\R\right)$,
\begin{equation}
\label{EQ:TreesEstInf}
\left\| \Xi_j (\vec{\psi} ) (T) \right\|_{\FL^\infty} \leq C^j T^{2j} \| \vec{\psi} \|^{(k-1)j-1}_{\FLv^1} \| \vec{\psi} \|^{2}_{\H^0}. 
\end{equation}
\end{lemma}

\begin{proof}
Let $\TT \in \BT (j)$. Using the same tree structure argument as for Lemma \ref{LEM:NumberOfTrees}, there exist $k$ trees $\TT_1 \in \BT(j_1), \dots,  \TT_k \in\BT(j_k)$, with $j_1 + \cdots + j_k = j-1$, such that the root nodes of $\TT_1, \dots, \TT_1$ are the children of the root node of $\TT$. Thus, we can write
$$
\Psi(\TT) (\vec{\phi}) = \I_k (\Psi(\TT_1) (\vec{\phi}), \cdots , \Psi(\TT_k) (\vec{\phi}) ),
$$

\noi and $\Psi(\TT) (\vec{\phi})$ consists essentially of $j = \abs{\TT^0}$ iterations of the Duhamel operator $\I_k$ with $(k-1)j + 1$ times the term $S(t)\vec{\phi}$ as arguments. 

Meanwhile, we deduce from \eqref{EQ:linearOp} that $S(t)$ is unitary in $\FL^1$ for any $ 0 < T \leq 1$, and, since $\abs{\sin y} \leq \abs{y}$ for any $y \in \R$ and $\ld(\xi) \leq 1$ for any $\xi \in \R$, the algebra property of $\FL^1$ gives
\begin{equation}
\label{EQ:EstDuhamelOp}
\| \I_k [u_1, \dots, u_k] \|_{C_T \FL^1} \leq \int^T_0 \abs{T-t'} \abs{\ld(\xi)}^2 \| u_1 \cdots u_k \|_{C_T \FL^1} \mathrm{d}t' \leq \frac 12 T^2 \prod^k_{j=1}  \| u_j \|_{C_T \FL^1}
\end{equation}

\noi Hence, \eqref{EQ:TreesEst1} follows from an induction argument and Lemma \ref{LEM:NumberOfTrees}. Besides, Young's inequality and a similar argument give \eqref{EQ:TreesEstInf}.

\end{proof}

Let us now use Lemma \ref{LEM:TreesEstFL} to prove local well-posedness of \eqref{imBq} in the Wiener algebra and to justify the power series expansion.

\begin{lemma}
\label{LEM:ExistenceOfSolution}
Let $M > 0$. Then, for any time $T$ such that $0 < T \ll \min(M^{-\frac{k-1}{2} }, 1)$ and $\vec{u}_0 \in \FLv^1$ with $\| \vec{u}_0 \|_{\FLv^1} \leq M$, 
\begin{enumerate}
	\item there exists a unique solution $u \in \C ([0,T]; \FL^1 (\R))$ to \eqref{imBq} satisfying $(u, \dt u) = \vec{u}_0$.
	\item Moreover, u can be expressed as
\begin{equation}
\label{EQ:PowerSeriesSol}
u = \sum_{j = 0}^\infty \Xi_j (\vec{u}_0) = \sum_{j = 0}^\infty \sum_{\TT \in \BT (j)} \Psi(\TT)(\vec{u}_0)
\end{equation}
\end{enumerate}

\end{lemma}

\begin{proof}

Let us first prove our problem is locally well-posed in $\FL^1$. We define the functional $\G$ by
$$
\G [u] (t) \coloneqq S(t)\vec{u}_0 + \I_k (u)(t)
$$

\noi for any $t \in [0,T]$. Then, using the unitarity of $S(t)$ and \eqref{EQ:EstDuhamelOp}, we have
$$
\left\| \G[u] \right\|_{C_T \FL^1} \leq \| \vec{u}_0 \|_{\FLv^1} + \frac 12 T^2 \| u \|_{\FL^1}^k
$$

\noi Using the multilinearity of $\I_k$, we ensure, for $0 < T \leq 1$ such that $\frac 12 T^2 M^{k-1} \ll 1$, that $\G$ is a strict contraction on the ball
$$
B_{2M} \coloneqq \{ v \in C([0,T], \FL^1 (\M)), \| v \|_{C_T \FL^1} \leq 2M\}.
$$

\noi Then, the contraction mapping theorem and an a posteriori continuity argument proves the local well-posedness. Let us move onto the power series expansion. 

Fix $\eps >0$. We choose $0 < T \leq 1$ such that $\frac 12 T^2 M^{k-1} \ll 1$. Then, from \eqref{EQ:TreesEst1}, the sum in \eqref{EQ:PowerSeriesSol} converges absolutely in $C([0,T], \FL^1 (\M))$. Let us denote 
$$
U_J = \sum_{j = 0}^J \Xi_j (\vec{u}_0) \qquad \text{ and } \qquad U = \sum_{j = 0}^\infty \Xi_j (\vec{u}_0).
$$

\noi There exists $J_1 \geq 0$ such that 
\begin{equation}
\label{EQ:PfPowerSeries1}
\| U - U_J \|_{C_T \FL^1} < \frac \eps3
\end{equation}

\noi for any $J \geq J_1$. In particular, this implies that $U$ and $U_J$ belong to the ball $B_{2M}$. Then, by continuity of $\G$ as a map from $B_{2M}$ into itself, \eqref{EQ:PfPowerSeries1} implies there exists $J_2 \leq 0$ such that, for any $J \geq J_2$,
\begin{equation}
\label{EQ:PfPowerSeries2}
\| \G[U] - \G[U_J] \|_{C_T \FL^1} < \frac \eps3.
\end{equation}

\noi All that is left now is to estimate $U_J - \G[U_J]$. Fix an integer $J \geq 1$. Then, from the tree structure argument we already used in the proof of Lemma \ref{LEM:TreesEstFL}, we get
\begin{align*}
U_J - \G[U_J] & = \sum_{j = 1}^J \Xi_j (\vec{u}_0) - \sum_{0 \leq j_1, \dots, j_k \leq J} \I_k \left( \Xi_{j_1}(\vec{u}_0) , \cdots,  \Xi_{j_k} (\vec{u}_0) \right) \\
	& = - \sum^{kJ}_{l = J} \sum_{\substack{0 \leq j_1 , \dots, j_k \leq J, \\ j_1 + \cdots + j_k = l}} \I_k \left( \Xi_{j_1}(\vec{u}_0) , \cdots, \Xi_{j_k} (\vec{u}_0) \right).
\end{align*}

\noi Now a crude estimation of the sums, along with \eqref{EQ:EstDuhamelOp} and \eqref{EQ:TreesEst1}, give
\begin{align*}
\| U_J - \G[U_J] \|_{C_T \FL^1} & \leq \frac 12 T^2 \sum^{kJ}_{l = J} \  \sum_{\substack{0 \leq j_1 , \dots, j_k \leq J, \\ j_1 + \cdots + j_k = l}} \   \prod^k_{m=1} \| \Xi_{j_m}(\vec{u}_0) \|_{C_T \FL^1}  \\
	& \leq \frac 12 T^2 \sum^{kJ}_{l = J} \  \sum_{\substack{0 \leq j_1 , \dots,  j_k \leq J, \\ j_1 + \cdots + j_k = l}} \ \prod^k_{m= 1} C^{j_m} T^{2j_m} \| \vec{u}_0 \|^{(k-1)j_m + 1}_{\FLv^1} \\
	& \leq \frac 12 T^2 J^k M^k \sum^{\infty}_{l = J} (C T^2 M^{k-1})^l
\end{align*}

\noi Since we assumed $0 < T \ll \min(M^{-\frac{k-1}{2} }, 1)$, the sum converges and the right-hand side is bounded by $\frac 12 T^2 J^k M^k (C T^2 M^{k-1})^J$, which tends to $0$ as $J$ tends to infinity. Thus, there exists $J_3 \geq 1$ such that for every $J \geq J_3$,
\begin{equation}
\label{EQ:PfPowerSeries3}
\| U_J - \G[U_J] \|_{C_T \FL^1} < \frac \eps3.
\end{equation}

\noi Now, for any $J \geq \max(J_1, J_2, J_3)$, \eqref{EQ:PfPowerSeries1}, \eqref{EQ:PfPowerSeries2} and \eqref{EQ:PfPowerSeries3} imply
$$
\| U - \G[U] \|_{C_T \FL^1} \leq \| U - U_J \|_{C_T \FL^1} + \| U_J - \G[U_J] \|_{C_T \FL^1} + \| \G[U] - \G[U_J] \|_{C_T \FL^1}  < \eps.
$$

\noi Therefore, $U$ is a fixed point of $\G$ and this ends the proof by uniqueness.

\end{proof}

\section{Norm inflation for IBq}
\label{Sec3}

In this section, we present first the proof of Theorem \ref{THM:main}. Actually, our main goal is to prove the following proposition:

\begin{proposition}
\label{THM:main2}
Let $k \geq 2$ and $s < 0$. Fix $(u_0 , u_1) \in \H^s (\M)$. Then, given any $\eps > 0$, there exists a solution $u_\eps$ to \eqref{imBq} with nonlinearity $f(u) = u^k$ on $\M$ and $t_\eps \in (0, \eps)$ such that
$$
\| (u_\eps (0), \dt u_\eps (0)) - (u_0 , u_1) \|_{\H^s} < \eps \quad \text{ and } \quad \| u_\eps (t_\eps) \|_{H^s} > \eps^{-1}.
$$
\end{proposition}

Indeed, once Proposition \ref{THM:main2} is proved, the proof of Theorem \ref{THM:main} follows in the same way, with a slight modification that we treat separately for more clarity. See Remark \ref{REM:ILOR}.

To do so, suppose first that we proved the following proposition:

\begin{proposition}
\label{PROP:final2}
Let $k \geq 2$ and $s < 0$. Fix $(u_0, u_1) \in \mathcal{S} ( \M) \times \mathcal{S}(\M)$. Then, for any $n \in \NB$, there exists a solution $u_n$ to \eqref{imBq} with nonlinearity $f(u) = u^k$ and $t_n \in (0, \frac 1n )$ such that
\begin{equation}
\label{EQ:propFinal}
\| (u_n (0), \dt u_n (0)) - (u_0 , u_1) \|_{\H^s} < \frac 1n \quad \text{ and } \quad \| u_n (t_n ) \|_{H^s} > n.
\end{equation}
\end{proposition}

\noi Then, let us fix $\eps > 0$ and choose $n \in \NB$ such that $n > \eps^{-1}$. According to Proposition \ref{PROP:final2}, there exists a solution $u_n$ to \eqref{imBq} and a time $t_n \in (0, \frac 1n ) \subset (0, \eps)$ such that 
$$
\| (u_n (0), \dt u_n (0)) - (u_0 , u_1) \|_{\H^s} < \frac 1n < \eps \quad \text{ and } \quad \| u_n (t_n ) \|_{H^s} > n > \eps^{-1}.
$$

\noi Therefore, Proposition \ref{THM:main2} follows from Proposition \ref{PROP:final2} and the density of $\mathcal{S}(\M)$ in $H^s (\M)$. Similarly, Theorem \ref{THM:main} follows from the following proposition:

\begin{proposition}
\label{PROP:final}
Let $k \geq 2$, $\s \in \R$ and $s < 0$. Fix $(u_0 , u_1) \in \mathcal{S} ( \M) \times \mathcal{S}(\M)$. Then, given any $n \in \NB$, there exists a solution $u_n$ to \eqref{imBq} with nonlinearity $f(u) = u^k$ on $\M$ and $t_n \in (0, \frac 1n )$ such that
$$
\| (u_n (0), \dt u_n (0)) - (u_0 , u_1) \|_{\H^s} < \frac 1n \quad \text{ and } \quad \| u_n (t_n) \|_{H^\s} > n.
$$
\end{proposition}

Consequently, the rest of this paper is devoted to the proofs of Proposition \ref{PROP:final2} and Proposition \ref{PROP:final}. In the following, we fix some $\vec{u}_0 \in \mathcal{S} ( \M) \times \mathcal{S}(\M)$.

\subsection{Multilinear estimates}

In this subsection, we establish some multilinear estimates that are essentials to our proof.

Given $n \in \NB$, let us fix some $N = N(n) \gg 1$, $R = R(N) \gg 1$ and $ 1 \ll A = A(N) \ll N$ to be determined later. We define $\vec{\phi}_n \coloneqq (\phi_{n} , 0)$ by setting
\begin{equation}
\label{DEF:phi_n}
\widehat{\phi_n} = R \chi_\O
\end{equation}

\noi where 
\begin{equation}
\label{DEF:Omega}
\O = \bigcup_{\eta \in \Si} (\eta + Q_A)
\end{equation}

\noi with $Q_A = [-\frac A2 , \frac A2]$ and
\begin{equation}
\label{DEF:Sigma}
\Si = \{ -2N, -N, N, 2N\}.
\end{equation}

\noi Note that the condition $A \ll N$ ensures the union \eqref{DEF:Omega} is disjoint. Besides, observe that \eqref{DEF:phi_n}, \eqref{DEF:Omega} and \eqref{DEF:Sigma} imply for any $s \in \R$
\begin{equation}
\label{EQ:EstPhi_n}
\| \vec{\phi}_n \|_{\FLv^1} \sim RA \quad \text{ and } \quad \| \vec{\phi}_n \|_{\H^s} \sim RN^s A^{1/2}.
\end{equation}

\noi We define finally $\vec{u}_{0,n} \coloneqq \vec{u}_0 + \vec{\phi}_n$. Suppose $N$, $R$ and $A$ satisfy 
\begin{equation}
\label{EQ:CondFL1u}
\| \vec{u}_0 \|_{\FLv^1} \ll RA.
\end{equation}

\noi Therefore, for each $n \in \NB$, provided 
\begin{equation}
\label{EQ:CondOnT}
0 < T \ll \min( (RA)^{-\frac{k-1}{2}} , 1),
\end{equation}

\noi Lemma \ref{LEM:ExistenceOfSolution} implies there exists a unique solution $u_n \in C([0,T], \FL^1(\M))$ to \eqref{imBq} with $(u_n , \dt u_n) |_{t=0} = \vec{u}_{0,n}$ and admitting the power series expansion:
\begin{equation}
\label{EQ:PowerSeriesSolRankn}
u_n = \sum_{j = 0}^\infty \Xi_j (\vec{u}_{0,n}) = \sum_{j = 0}^\infty \Xi_j (\vec{u}_{0} + \vec{\phi}_n) .
\end{equation}

\noi The purpose of this subsection is then to estimate the terms of the power series on the right-hand side of \eqref{EQ:PowerSeriesSolRankn}. But first, let us recall the following lemma:

\begin{lemma}
\label{LEM:ConvolutionIneq}
Let $a,b \in \R$ and $A > 0$, then we have
$$
C A \chi_{a + b + Q_A } (\xi) \leq  \chi_{a + Q_A } \ast \chi_{b + Q_A } (\xi) \leq \widetilde{C} A \chi_{a + b + Q_{2A}}(\xi )
$$

\noi where $C, \widetilde{C} > 0$ are constants.
\end{lemma}

The proof of the following lemma is essentially the same as in \cite[Lemma 3.2]{FO} and is included for completeness.

\begin{lemma}
\label{LEM:MultilinearEst}

For any $s <0$, $t \in [0, T]$ and $j \in \NB$, the following estimates hold:
\begin{align}
\| \vec{u}_{0,n} - \vec{u}_0 \|_{\H^s} & \sim RN^s A^{1/2}, \label{EQ:MultilinearEst1} \\
\| \Xi_0 (\vec{u}_{0,n})(t) \|_{H^s} & \les 1 + RA^{1/2} N^s, \label{EQ:MultilinearEst2} \\
\| \Xi_1 (\vec{u}_{0,n})(t) - \Xi_1 (\vec{\phi}_n) (t) \|_{H^s} & \les t^2 R^{k-1}A^{k-1} \| \vec{u}_0 \|_{\H^0}, \label{EQ:MultilinearEst3} \\
\| \Xi_j (\vec{u}_{0,n})(t) \|_{H^s} & \les C^j t^{2j} R^{(k-1)j}A^{(k-1)j} \left( \| \vec{u}_0 \|_{\H^0} +  R g_s (A) \right), \label{EQ:MultilinearEst4}
\end{align}

\noi where $g_s (A)$ is defined by 
\begin{equation}
\label{DEF:gs}
g_s (A) \coloneqq
\begin{cases}
 1 & \textup{if } s<-\frac{1}{2}, \\
 \left( \log A \right)^{\frac{1}{2}}& \textup{if } s=-\frac{1}{2}, \\
 A^{\frac{1}{2}+s} & \textup{if }  s>-\frac{1}{2}
\end{cases}
\end{equation}

\end{lemma}

\begin{proof}
The proofs of \eqref{EQ:MultilinearEst1} and \eqref{EQ:MultilinearEst2} follow directly from $\vec{u}_{0,n} = \vec{u}_0 + \vec{\phi}_n$, \eqref{EQ:EstPhi_n}, the unitarity of $S(t)$ for $t \leq 1$ and the fact that $\vec{u}_0$ is fixed, implying $\|\vec{u}_0\|_{\H^s} \les 1$. Besides, the definition \eqref{DEF:Xi_j} of $\Xi_j$ and the multilinearity of $\I_k$ imply
$$
\Xi_1 (\vec{u}_{0,n})(t) - \Xi_1 (\vec{\phi}_n) (t) = \sum\limits_{(v_1, \dots, v_k) \in E} \mathcal{I}_k (v_1, \cdots, v_k)(t)
$$

\noi where $E$ is the subset of $\{ S(t)\vec{u}_0 , S(t)\vec{\phi}_n \}^k$ such that at least one of the choice is $S(t)\vec{u}_0$. Thus, since $s<0$, Young's inequality, \eqref{EQ:EstDuhamelOp} and the unitarity of $S(t)$ imply
\begin{align*}
\| \Xi_1 (\vec{u}_{0,n})(t) - \Xi_1 (\vec{\phi}_n) (t) \|_{H^s} & \les  t^2 \|\vec{u}_0 \|_{\H^0} \left(  \| \vec{u}_0 \|^{k-1}_{\FLv^1} + \| \vec{\phi}_n \|^{k-1}_{\FLv^1} \right) \\
	& \les t^2 \|\vec{u}_0 \|_{\H^0} \left(  \| \vec{u}_0 \|^{k-1}_{\FLv^1} + R^{k-1}A^{k-1} \right)
\end{align*}

\noi which, combined with assumption \eqref{EQ:CondFL1u}, proves \eqref{EQ:MultilinearEst3}.

For the last inequality, we split the left-hand side into two terms: 
\begin{equation}
\label{EQ:SplitXi_j}
\Xi_j (\vec{u}_{0,n})(t) = \left( \Xi_j (\vec{u}_{0,n})(t) - \Xi_j (\vec{\phi}_{n})(t) \right) + \Xi_j (\vec{\phi}_{n})(t)
\end{equation}

\noi which we will estimate separately.

{\bf Part 1:} $\F\big(\Xi_j (\vec{\phi}_{n})(t)\big)$ essentially consists of $(k-1)j + 1$ convolutions of terms of the form $\F(S(t) \vec{\phi}_n)$. According to \eqref{EQ:linearOp} and \eqref{DEF:phi_n}, the support of each of these terms is contained within at most $4$ disjoint cubes of volume approximately $A$. Thus, Lemma \ref{LEM:ConvolutionIneq} and a countability argument show the support of $ \F( \Xi_j (\vec{\phi}_{n})(t) )$ is contained within at most $4^{(k-1)j+1}$ cubes of volume approximately $A$. Therefore, \eqref{DEF:Xi_j} and Lemma \ref{LEM:NumberOfTrees} imply there exist $c,C > 0$ such that
$$
|\supp \F( \Xi_j (\vec{\phi}_{n})(t) ) | \leq C^j A \leq c |C^j Q_A |.
$$

\noi Since $s < 0$, $\jb{\xi}^s$ is decreasing in $\abs{\xi}$ and Young's inequality, \eqref{EQ:TreesEstInf} and the unitarity of $S(t)$ yield
\begin{align*}
\| \Xi_j (\vec{\phi}_{n})(t) \|_{H^s} & \leq \| \jb{\xi}^s \|_{L^2 (\supp \F( \Xi_j (\vec{\phi}_{n})(t) ) ) } \| \Xi_j (\vec{\phi}_{n})(t) \|_{\FL^\infty} \\
	& \les \| \jb{\xi}^s \|_{L^2 (c C^j Q_A) } C^j t^{2j} \| \vec{\phi}_n \|^{(k-1)j-1}_{\FLv^1} \| \vec{\phi}_n \|^2_{\H^0}.
\end{align*}

\noi Since $\| \jb{\xi}^s \|_{L^2 (c C^j Q_A) } \les g_s(A)$ with $g_s$ defined as in \eqref{DEF:gs}, we get from \eqref{EQ:EstPhi_n}
\begin{equation}
\label{EQ:EstXi_jPhi}
\| \Xi_j (\vec{u}_{0,n})(t) \|_{H^s} \les C^j t^{2j} R^{(k-1)j}A^{(k-1)j} R g_s (A).
\end{equation}

{\bf Part 2:} On the other hand, since $\F\big(\Xi_j (\vec{u}_{0})(t) - \Xi_j (\vec{\phi}_{n})(t)\big)$ is essentially a sum of terms made of $j$ integrals in time and $(k-1)j + 1$ convolutions of terms of the form $\F(v)$, with $v \in \{ S(t)\vec{u}_0 , S(t)\vec{\phi}_n \}$, a similar argument as for \eqref{EQ:MultilinearEst3} along with $s < 0$ yield
\begin{align*}
\|  \Xi_j (\vec{u}_{0,n})(t) - \Xi_j (\vec{\phi}_{n})(t) \|_{H^s} & \leq \|  \Xi_j (\vec{u}_{0,n})(t) - \Xi_j (\vec{\phi}_{n})(t) \|_{L^2} \\
	& \les C^j t^{2j} \| \vec{u}_0 \|_{\H^0} \left( \| S(t)\vec{u}_0\|_{\FL^1} + \| S(t)\vec{\phi}_n \|_{\FL^1} \right)^{(k-1)j}.
\end{align*}

\noi Thus, the unitarity of $S(t)$, \eqref{EQ:EstPhi_n} and \eqref{EQ:CondFL1u} give
\begin{equation}
\label{EQ:EstXi_jRem}
\|  \Xi_j (\vec{u}_{0,n})(t) - \Xi_j (\vec{\phi}_{n})(t) \|_{H^s} \les C^j  t^{2j} \| \vec{u}_0 \|_{\H^0} (RA)^{(k-1)j}.
\end{equation}

\noi Now, \eqref{EQ:MultilinearEst4} follows from triangular inequality, \eqref{EQ:EstXi_jPhi} and \eqref{EQ:EstXi_jRem}.

\end{proof}

With the help of Lemma \ref{LEM:ConvolutionIneq}, we can also state the following crucial proposition. This proposition exploits the high-to-low energy transfer mentioned before to identify the first multilinear term in the Picard expansion as the one responsible for the instability in Proposition \ref{PROP:final2} and in Proposition \ref{PROP:final}.

\begin{proposition}
\label{PROP:EstSecondPicard}
Let $s < 0$, $\vec{\phi}_n$ defined as in equation \eqref{DEF:phi_n}, $T \ll 1$ and $A \ll N $. Then, there exists a constant $C > 0$ such that
\begin{equation}
\label{EQ:EstSecondPicard}
\| \Xi_1 (\vec{\phi}_n) (T) \|_{ H^s} \ges C^{k-1} R^k T^2 A^{k - \frac 12 + s}.
\end{equation}

\end{proposition}

\begin{proof}
To simplify notation, we write
$$
\Ld \coloneqq \left\{ (\xi_1, \dots, \xi_k) \in \widehat{\M}^k \colon \sum^k_{j=1} \xi_j = \xi \right\} \quad \text{ and } \quad \mathrm{d}\xi_\Ld \coloneqq \mathrm{d}\xi_1 \cdots \mathrm{d}\xi_{k-1}.
$$
Using \eqref{DEF:phi_n} and a product-to-sum formula, we have
\begin{align*}
\F[ \Xi_1 (\vec{\phi}_n)] (T, \xi) = R^k \sum\limits_{(\eta_1, \dots, \eta_k) \in \Si^k} & \int^T_0 \sin((T - t') \ld(\xi)) \ld(\xi) \\
	& \times\int_\Ld \prod^k_{j=1} \cos(t' \ld(\xi_j)) \chi_{\eta_j + Q_A} (\xi_j) \mathrm{d}\xi_\Ld \mathrm{d}t'.
\end{align*}

\noi We split the set $\Si^k$, and hence the sum, into two parts $\Si_1$ and $\Si_2$ where
$$
\Si_1 = \left\{ (\eta_1, \dots, \eta_k) \in \Si^k \colon \eta_1+ \cdots +\eta_k = 0 \right\}
$$

\noi and $\Si_2 = \Si^k \backslash \Si_1$. Note that, for any integer $k \geq 2$, $\Si_1$ and $\Si_2$ are both non-empty. We denote then
\begin{equation}
\label{EQ:ProofSecondPicardSplit}
\F[ \Xi_1 (\vec{\phi}_n)] (T, \xi) = R^k \left( I_1 (T, \xi) + I_2 (T,\xi) \right)
\end{equation}

\noi with, for $r = 1,2$,
$$
I_r (T, \xi) = \sum\limits_{(\eta_1, \dots, \eta_k) \in \Si_r} \int^T_0 \sin((T - t') \ld(\xi)) \ld(\xi) \int_\Ld \prod^k_{j=1} \cos(t' \ld(\xi_j)) \chi_{\eta_j + Q_A} (\xi_j) \mathrm{d}\xi_\Ld \mathrm{d}t'.
$$

Since $T \ll 1$, $\ld(\xi) \leq 1$, $\sin(x) \sim x$ for $\abs{x} \leq 1$ and $\cos(x) \geq \frac 12$ for $\abs{x} \ll 1$, we have
\begin{align*}
I_1 (T, \xi) & \ges \sum\limits_{(\eta_1, \dots, \eta_k) \in \Si_1} \frac 12 T^2 \ld(\xi)^2 \int_\Ld \frac{1}{2^k} \prod^k_{j=1} \chi_{\eta_j + Q_A}(\xi_j) \mathrm{d}\xi_\Ld \\
	& \ges \frac{1}{2^{k+1}} T^2 C^{k-1}_1 A^{k-1} \ld(\xi)^2 \chi_{Q_A} (\xi),
\end{align*}

\noi where the second inequality comes from $\abs{\Si_1} > 1$ and Lemma \ref{LEM:ConvolutionIneq}. Computing the $H^s$ norm, we get
\begin{equation}
\label{EQ:PicardSecondLowFreq}
\| I_1(T) \|_{H^s} \ges C^{k-1}_1 T^2 A^{k - \frac 12 + s}.
\end{equation}

Meanwhile, since $\sin(x) \sim x$ for $\abs{x} < 1$, $\abs{ \cos(x)}\leq 1$, and $\ld(\xi) \leq 1$, we have 
$$
\abs{I_2 (T, \xi)}  \les \sum\limits_{(\eta_1, \dots, \eta_k) \in \Si_2} \frac 12 T^2 \int_\Ld \prod^k_{j=1} \chi_{\eta_j + Q_A}(\xi_j) \mathrm{d}\xi_\Ld.
$$

\noi Fix $(\eta_1 , \dots, \eta_k) \in \Si_2$. Since $\eta_1 + \cdots + \eta_k \neq 0$, \eqref{DEF:Sigma} implies there exists $m \in \Z \backslash \{ 0 \}$ such that 
$$
\eta_1 + \cdots + \eta_k = m N.
$$

\noi Lemma \ref{LEM:ConvolutionIneq} implies then
\begin{align*}
\bigg\| \int_\Ld \prod^k_{j=1} \chi_{\eta_j + Q_A}(\xi_j) \mathrm{d}\xi_\Ld \bigg\|_{H^s} & \les C^{k-1}_1 A^{k-1} \| \chi_{ mN + Q_{kA}} (\xi)\|_{H^s} \\
	& \les C^{k-1}_1 A^{k-\frac 12} N^s.
\end{align*}

\noi Since $\abs{\Si_2} \leq \abs{\Si^k} \leq 4^k$, we get
\begin{equation}
\label{EQ:PicardSecondHighFreq}
\| I_2 (T) \|_{H^s}  \les C^{k-1}_1 T^2 A^{k - \frac 12} N^s .
\end{equation}

\noi Combining \eqref{EQ:PicardSecondLowFreq} and \eqref{EQ:PicardSecondHighFreq} with triangular inequality, we have
\begin{align*}
\| \Xi_1 (\vec{\phi}_n) (T) \|_{ H^s} & \geq R^k \big| \| I_1(T) \|_{H^s} - \| I_2 (T) \|_{H^s} \big| \\
	& \ges R^k C^{k-1}_1 T^2 A^{k - \frac 12} \big(  A^{s} -  N^s \big)
\end{align*}

\noi which, with the assumption $A \ll N$, proves our result.

\end{proof}

\subsection{Proof of Proposition \ref{PROP:final2}}

In this subsection we give the proof of Proposition \ref{PROP:final2}. We claim it suffices to show that, given $n \in \NB$, the following inequalities hold:
\begin{align*}
\textup{(i)} & \quad RA^{\frac 12} N^s < \frac 1n, \\
\textup{(ii)} & \quad T^2 R^{k-1} A^{k-1} \ll 1, \\
\textup{(iii)} & \quad \| \vec{u}_0 \|_{\H^0} \ll R A^{\frac 12 + s} \text{ and } \| \vec{u}_0\|_{\FLv^1} \ll RA, \\
\textup{(iv)} & \quad T^4  (RA)^{2(k-1)} R g_s (A) \ll T^2 R^k A^{k - \frac 12 + s}, \\
\textup{(v)} & \quad T^2 R^k A^{k - \frac 12 + s} \gg n, \\
\textup{(vi)} & \quad A \ll N
\end{align*}

\noi for some particular $R$, $T$ and $N$, all depending on $n$.

Let us show conditions (i) through (vi) prove Proposition \ref{PROP:final2}. First, condition (i) along with \eqref{EQ:MultilinearEst1} verifies the first estimate in \eqref{EQ:propFinal}. Besides, conditions (ii) and the second of (iii), along with Lemma \ref{LEM:ExistenceOfSolution}, prove existence and uniqueness of a solution $u_n$ in $C([0,T], \FL^1)$ with $(u_n , \dt u_n) |_{t=0} = \vec{u}_{0,n}$, as well as the convergence of the power series \eqref{EQ:PowerSeriesSolRankn}. Furthermore, these conditions, along with \eqref{EQ:MultilinearEst4}, yield
\begin{align*}
\sum^{\infty}_{j = 2} \| \Xi_j (\vec{u}_{0,n})(T) \|_{H^s} & \les T^4 (RA)^{2(k-1)} \left(\sum^{\infty}_{j = 0} ( C T^2 R^{k-1}A^{k-1} )^j \right) \left( \| \vec{u}_0 \|_{\H^0} + Rg_s (A) \right) \\
	& \les T^4 (RA)^{2(k-1)} R g_s (A).
\end{align*}

\noi Then, using \eqref{EQ:PowerSeriesSolRankn}, Lemma \ref{LEM:MultilinearEst} and Proposition \ref{PROP:EstSecondPicard} - which is applicable by condition (vi) - conditions (i) through (v) give:
\begin{align*}
\| u_n (T) \|_{H^s} & \geq \|\Xi_1 (\vec{\phi}_{n})(T) \|_{H^s} \\
	& \qquad - \bigg\| \Xi_0 (\vec{u}_{0,n})(T) + \left( \Xi_1 (\vec{u}_{0,n})(T) - \Xi_1 (\vec{\phi}_{n})(T) \right)  +  \sum^{\infty}_{j = 2} \Xi_j (\vec{u}_{0,n})(T) \bigg\|_{H^s} \\
	 & \ges R^k T^2 A^{k - \frac 12 + s} - 1 - RA^{1/2} N^s - T^2 R^{k-1}A^{k-1} \| \vec{u}_0 \|_{\H^0}  -  T^4 (RA)^{2(k-1)} R g_s (A) \\
	& \sim R^k T^2 A^{k - \frac 12 + s} \gg n.
\end{align*}

Thus, this verifies the second estimate in \eqref{EQ:propFinal} at time $t_n \coloneqq T$. Finally, a suitable choice of $T$ in terms of $N = N(n)$ ensures $t_n \in (0, \frac 1n)$, for $N(n)$ sufficiently large. See details below. So ends the proof of Proposition \ref{PROP:final2}.

Consequently, it only remains to verify the conditions (i) through (vi). To do so, we express $A$, $R$ and $T$ in terms of $N$. More precisely, let us choose
$$
A = 10, \quad R = N^{-s-\dl} \quad \text{ and } \quad T = N^{\frac{k-1}{2}(s + \frac \dl2)}
$$

\noi where $\dl$ is sufficiently small, namely $0 < \dl < \min \big(1, -\frac{2}{k+1}s \big)$. Since $A$ is a constant, condition (vi) is trivially satisfied and, since $\vec{u}_0$ is fixed, condition (iii) reduces to
$$
R \gg 1 $$

\noi which is true from $-s - \dl > 0$, for $N$ sufficiently large. Besides, for $N$ sufficiently large, we get
\begin{align*}
& RA^{\frac 12} N^s \sim N^{-\dl} \ll \frac 1n \\
& T^2 R^{k-1} A^{k-1} \sim N^{- \frac{k-1}{2} \dl} \ll 1 \\
& T^2 R^k A^{k - \frac 12 + s} \sim N^{-s - \frac{k+1}{2}\dl} \gg 1
\end{align*}

\noi which prove conditions (i), (ii) and (v). Lastly, condition (iv) is equivalent to 
$$
 T^2 R^{k-1} A^{k - 2}g_s(A) \ll A^{-\frac 12 + s}
$$

\noi which, in our setting, is equivalent to condition (ii) since $A$ is constant. Finally, since $\dl < -s$, observe that $T$ goes to $0$ as $N \to \infty$, so $T \in (0, \frac 1n )$ for $N$ sufficiently large. Taking
$$
N = n^{\frac 2 \dl}
$$
completes the proof since $\dl < 1$.

\begin{remark} \rm
\label{REM:ILOR}

In this remark, we use our previous arguments to prove infinite loss of regularity, namely Proposition \ref{PROP:final}. The idea is to use the same construction as before. Indeed, this construction allowed us to choose our parameter $A$ to be constant. In the following, we observe that the change of regularity between Proposition \ref{PROP:final2} and Proposition \ref{PROP:final} is only expressed in powers of $A$, so that this change has no major implications whatsoever.

Let $s < 0$, $\s \in \R$ and $\vec{u}_0 = ( u_0, u_1) \in \FLv^1$. Define $\vec{\phi}_n$ and $\vec{u}_{0,n}$ as before. There exists, for any given $n \in \NB$, a unique solution $u_n$ to \eqref{imBq} in $C([0,T], \FL^1)$, with $T$ satisfying \eqref{EQ:CondOnT}, such that $(u_n , \dt u_n)|_{t=0} = \vec{u}_{0,n}$. Moreover, $u_n$ can be expressed as the power series in \eqref{EQ:PowerSeriesSolRankn}. First, we claim that it suffices to consider the case $\s < s$. Indeed, in the case $\s \geq s$ the estimate on the initial data remains the same as for Proposition \ref{PROP:final2} and, by Sobolev embedding, we have 
$$
\| u_n (t_n) \|_{H^\s} \geq \| u_n (t_n) \|_{H^s} > n.
$$

Let us then assume $\s < s$. Lemma \ref{LEM:MultilinearEst}, \eqref{EQ:PowerSeriesSolRankn}, \eqref{EQ:CondOnT} and \eqref{EQ:CondFL1u} yield
\begin{align*}
\| u_n(T) - \Xi_1(\vec{\phi}_n) (T) \|_{H^\s} & \leq \| u_n(T) - \Xi_1(\vec{\phi}_n) (T) \|_{H^s} \\
	& \les 1 + RA^{1/2} N^s + T^2 (RA)^{k-1} \| \vec{u}_0 \|_{\H^0}  +  T^4 (RA)^{2(k-1)} R g_s (A)
\end{align*}

\noi while Proposition \ref{PROP:EstSecondPicard} gives
$$
\| \Xi_1 (\vec{\phi}_n) (T) \|_{ H^\s} \ges R^k T^2 A^{k - \frac 12 + \s}.
$$

\noi Therefore, the same arguments as before allow us to say that, to prove Proposition \ref{PROP:final}, it suffices to verify the following hold:
\begin{align*}
\textup{(i)} & \quad RA^{\frac 12} N^s < \frac 1n, \\
\textup{(ii)} & \quad T^2 R^{k-1} A^{k-1} \ll 1, \\
\textup{(iii)} & \quad \| \vec{u}_0 \|_{\H^0} \ll R A^{\frac 12 + \s} \text{ and } \| \vec{u}_0\|_{\FLv^1} \ll RA, \\
\textup{(iv)} & \quad T^4  (RA)^{2(k-1)} R g_s (A) \ll T^2 R^k A^{k - \frac 12 + \s}, \\
\textup{(v)} & \quad T^2 R^k A^{k - \frac 12 + \s} \gg n, \\
\textup{(vi)} & \quad A \ll N.
\end{align*}

\noi Note that, compared to the conditions in the proof of Proposition \ref{PROP:final2}, the only changes are in conditions (iii), (iv) and (v), where the power of $A$ changes. Yet, if we choose $A$ to be constant again, the rest of the reasoning stays the same. Hence, the choices
$$
A = 10, \quad R = N^{-s-\dl}, \quad T = N^{\frac{k-1}{2}(s + \frac \dl2)}, \quad \text{ and }  \quad N = n^{\frac 2 \dl}
$$

\noi prove also Proposition \ref{PROP:final}.
\end{remark}

\begin{remark} \rm
\label{RK:ProofGenIBq}
The proof we gave applies directly to the study of the multi-dimensional generalized improved Boussinesq equation \eqref{generalizedimBq}. The only major change to make is in the definition of the functions $\phi_n$. Indeed, let us denote $e_1 = (1, 0, \cdots, 0) \in \R^d$, we then define $\phi_n$ by
$$
\widehat{\phi}_n (\xi) = R \chi_\O = R \sum_{\eta \in \Si} \chi_{\{ \eta e_1 + Q_A \}}
$$

\noi with $\Si$ as in \eqref{DEF:Sigma} and $Q_A = [-\frac A2 , \frac A2]^d$. Then, the same arguments are still applicable. In particular, observe that Lemma \ref{LEM:ConvolutionIneq} can be rephrased as:

\begin{lemma}
\label{LEM:ConvolutionIneqInRd}
Let $a,b \in \R^d$ and $A > 0$, then we have
$$
C_d A^d \chi_{a + b + Q_A } (\xi) \leq  \chi_{a + Q_A } \ast \chi_{b + Q_A } (\xi) \leq \widetilde{C}_d A^d \chi_{a + b + Q_{2A}}(\xi )
$$

\noi where $C_d, \widetilde{C}_d > 0$ are constants depending only on the dimension $d$.
\end{lemma}

Subsequently, the only changes caused from the change of dimension are expressed in powers of $A$. However, since $A$ is chosen to be constant, it does not change the argument, as for when we proved infinite loss of regularity from ``standard" norm inflation. Therefore, norm inflation with infinite loss of regularity for initial data $\vec{u}_0 \in \H^s(\R^d)$ with $s<0$ follows naturally from the same choice of parameters $A$, $R$, $T$ and $N$.

Similarly, our result would still apply on the same multidimensional problem on the torus $\T^d$, although such a problem has not been seen in the litterature yet, at least not to our knowledge.
\end{remark}

\begin{remark}\rm
\label{REM:NIinWs2infty}

In this remark, we show that our proof also implies Theorem~\ref{THM:mainLinfty}. The idea is to use exactly the same construction and to estimate the norm. Using the same notations, observe that we have
$$
\| u_n (t_n) \|_{W^{\s, 2, \infty}}  \geq \| u_n (t_n) \|_{H^\s} > n.
$$


\noi Therefore, we only need to show that 
$$
\| \vec{u}_0 - \vec{u}_{0,n} \|_{\W^{s,2,\infty}} < \frac 1n
$$

\noi which means that we want $\vec{u}_{0}$ and $\vec{u}_{0,n}$ to satisfy
$$
\| \vec{u}_0 - \vec{u}_{0,n} \|_{\H^{s}} < \frac 1n \quad \text{ and } \quad \| \vec{u}_0 - \vec{u}_{0,n} \|_{W^{s,\infty} \times W^{s,\infty}} < \frac 1n.
$$

\noi We already proved that the first estimate is satisfied, so we are only interested in the second one. Observe that $\vec{u}_0 - \vec{u}_{0,n} = \vec{\phi}_n = (\phi_n , 0)$ and $\phi_n$ is frequently supported on $\O$, defined by \eqref{DEF:Omega}. Then, we get by Cauchy-Schwarz and inverse Fourier transform
$$
| \jb{\nb}^s \phi_n | \leq \int_{\widehat{\M}} | \jb{\xi}^s \widehat{\phi}_n (\xi)| \mathrm{d}\xi = \int_{\O} \jb{\xi}^s |\widehat{\phi}_n (\xi)| \mathrm{d}\xi \leq 2A^{\frac 12} \| \phi_n \|_{H^s}
$$

\noi Since $A = 10$ in our proof, and $\| \vec{\phi}_n \|_{\W^{s, 2, \infty}} = \| \phi_n \|_{W^{s, 2, \infty}}$, this shows the desired result and Theorem~\ref{THM:mainLinfty}. Furthermore, combining this argument with Remark~\ref{RK:ProofGenIBq}, this result still holds in higher dimension.

\end{remark}

\appendix

\section{Norm inflation for other spaces}
\label{appendixA}

In this appendix, we come back to Remark \ref{REM:NIforFLMW} and show how to prove norm inflation with infinite loss of regularity at general initial data for equation \eqref{imBq} in Fourier-Lebesgue, modulation and Wiener amalgam spaces. However, we do not show the entire proof since it is mostly similar to the one we gave for Sobolev spaces. Indeed, we use the same construction and decomposition as before, so the first differences are in the estimates we find in Lemma \ref{LEM:MultilinearEst} and Proposition \ref{PROP:EstSecondPicard}. Fortunately we find that the new estimates are, under the condition $A = 10$, equivalent to the ones we had for Sobolev spaces, so the last part of our argument remains unchanged. Therefore, we just show quickly how to get these estimates. Again, a whole proof would be quite redundant with what was shown before, the arguments being essentially the same. Thus, we just point out the key differences. Before that, we introduce our spaces. Let us recall, at this point, that modulation and Wiener amalgam spaces were introduced by Feichtinger~\cite{Fei83}.

\subsection{New spaces and some preliminary results.}

First, let us recall that Fourier-Lebesgue spaces are defined in Definition \ref{DEF:FLspaces}. Now, we introduce the modulation and Wiener amalgam spaces.

For any $n \in \Z^d$, we define $Q_n = n + \left[- \frac 12, \frac 12 \right]^d$ so that $\R^d = \bigcup_{n \in \Z^d} Q_n$. Then, let $\rho \in \mathcal{S} (\R^d)$ such that $\rho \colon \R^d \to [0,1]$ and 
\begin{equation}
\rho(\xi) = \begin{cases}
1 \quad \text{ if } \quad \abs{\xi} \leq \frac 12 \\
0 \quad \text{ if } \quad \abs{\xi} \geq 1
\end{cases}
\end{equation}

\noi We also denote $\rho_n (\xi) = \rho (\xi - n)$ for any $n \in \Z^d$. Let us define
$$
\s_n = \frac{\rho_n}{\sum_{l \in \Z} \rho_l }
$$

\noi and $P_n f = \F^{-1} ( \s_n \F(f))$ for any suitable $f$. Then, we define the modulation spaces in the following way:

\begin{definition}[Modulation spaces] \rm
Let $s \in \R$ and $p,q \geq 1$. The Modulation space $M^{p,q}_s (\R^d)$ is the completion of the Schwartz class of functions $\mathcal{S} ( \R^d)$ with respect to the norm
$$
\| f \|_{M^{p,q}_s (\R^d)} = \big\| (1 +  \abs{n}^s) \|P_n f \|_{L^p_x (\R^d)} \big\|_{\ell^q_n (\Z^d)}.
$$
\end{definition}

\noi We also define Wiener amalgam spaces in the following way:

\begin{definition}[Wiener amalgam spaces] \rm
Let $s \in \R$ and $p,q \geq 1$. The Wiener amalgam space $W^{p,q}_s (\R^d)$ is the completion of the Schwartz class of functions $\mathcal{S} ( \R^d)$ with respect to the norm
$$
\| f \|_{W^{p,q}_s (\R^d)} = \big\| \|(1 +  \abs{n}^s)P_n f \|_{\ell^q_n (\Z^d)} \big\|_{L^p_x (\R^d)}.
$$
\end{definition}

\noi We have the following relations between modulation and Wiener amalgam spaces:

\begin{lemma}
\label{LEM:RelationsMW}
Let $1 \leq p,q, p_1, q_1, p_2, q_2 \leq \infty$ and $d \geq 1$ be any finite dimension. Then:
\begin{enumerate}
	\item For $p_1 \leq p_2$ and $q_1 \leq q_2$, we have 
\begin{equation}
\label{EQ:EmbeddingsM}
M^{p_1, q_1}_s (\R^d) \embeds M^{p_2, q_2}_s (\R^d),
\end{equation}
\noi and
\begin{equation}
\label{EQ:EmbeddingsW}
W^{p, q_1}_s (\R^d) \embeds W^{p, q_2}_s (\R^d),
\end{equation}
	\item for $q\leq p$, we have
\begin{equation}
\label{EQ:EmbeddingMintoW}
M^{p, q}_s (\R^d) \embeds W^{p, q}_s (\R^d),
\end{equation}
	\noi and for $p \leq q$ we have
\begin{equation}
\label{EQ:EmbeddingWintoM}
W^{p, q}_s (\R^d) \embeds M^{p, q}_s (\R^d),
\end{equation}
	\item and we finally have 
\begin{equation}
\label{EQ:RelationsMandW_end}
M^{p, \min(p,q)}_s (\R^d) \embeds W^{p, q}_s (\R^d) \embeds M^{p, \max(p,q)}_s (\R^d).
\end{equation}
\end{enumerate}
\end{lemma}

\noi The proof of this Lemma can be seen in \cite{BH}, but split into different proofs. For completeness and ease of reading, we include it here. However, we first need the following lemma, which proof can be seen in \cite[Lemma 6.1] {WHHG}:

\begin{lemma}
\label{LEM:lem6.1}
Let $\O$ be a compact subset of $\R^d$ with $d \geq 1$, such that $\mathrm{diam} \ \O < 2R$, with $R > 0$, and $1 \leq p \leq q \leq \infty$. Then, there exists a constant $C > 0$ depending only on $p$, $q$ and $R$ such that:
$$
\| f \|_{L^q (\R^d)} \leq C \| f \|_{L^p (\R^d)}
$$

\noi for any function $f \in L^p (\R^d)$ such that its Fourier transform $\widehat{f}$ is compactly supported in $\O$.
\end{lemma}

\begin{proof}[Proof of Lemma \ref{LEM:RelationsMW}]
The embeddings \eqref{EQ:EmbeddingsM} and \eqref{EQ:EmbeddingsW} follow from Lemma \ref{LEM:lem6.1} and the fact that, for any $p \leq q$, $\ell^{p} (\Z^d) \embeds \ell^{q} (\Z^d)$. 

Embeddings \eqref{EQ:EmbeddingMintoW} and \eqref{EQ:EmbeddingWintoM} follow from Minkowski's inequality. 

Finally, \eqref{EQ:RelationsMandW_end} is a combination of \eqref{EQ:EmbeddingsM}, \eqref{EQ:EmbeddingsW}, \eqref{EQ:EmbeddingMintoW}  and \eqref{EQ:EmbeddingWintoM}.
\end{proof}

\noi Besides, we also have the following algebra property:

\begin{lemma}
\label{LEM:algebraM21}
For any $d \geq 1$, the space $M^{2,1}_0 (\R^d)$ is a Banach algebra.
\end{lemma}

\begin{proof}
We want to show that, for any $u,v \in M^{2,1}_0 (\R^d)$, we have
$$
\| uv \|_{M^{2,1}_0 (\R^d)} = \sum_{n \in \Z^d} \| P_n (uv) \|_{L^2 (\R^d)} \leq \| u \|_{M^{2,1}_0 (\R^d)} \| v \|_{M^{2,1}_0 (\R^d)}.
$$

\noi Note that $u = \sum_{m \in \Z^d} P_m u $ and $v = \sum_{k \in \Z^d} P_k v$. Therefore
$$
P_n (uv) = \sum_{m,k \in \Z^d} P_n (P_m u P_k v ).
$$

\noi The idea then is to study the Fourier support of each of the terms $P_n (P_m u P_k v )$. Since, for any $\xi \in \R^d$,
$$
\F [ P_n (P_m u P_k v ) ] (\xi) = \s_n (\xi) \int_{\xi = \xi_1 + \xi_2} \s_m (\xi_1) \widehat{u}(\xi_1) \s_k (\xi_2) \widehat{v}(\xi_2) \mathrm{d}\xi_1
$$

\noi and, for any $N \in \Z^d$, $\supp \s_N \subset \{ \xi \in \R^d, \abs{\xi - N} \leq 1 \}$, we have then
$$
\supp \big[ (\s_m \widehat{u}) \ast (\s_k \widehat{v}) \big] \subset  \{ \xi \in \R^d, \abs{\xi - m - k} \leq 2 \}
$$

\noi and
$$
\F  [ P_n (P_m u P_k v ) ] \equiv 0 \quad \text{ if } \quad \abs{n - m - k} > 3.
$$

\noi Plancherel's identity, H\"older's inequality and \eqref{EQ:EmbeddingsM} yield then:
\begin{align*}
\|  P_n (P_m u P_k v ) \|_{L^2 ( \R^d)} & \leq \|  P_m u P_k v \|_{L^2(\R^d)} \chi_{(\abs{n- k - m} \leq 3)}, \\
	& \leq \| P_m u \|_{L^\infty (\R^d)} \| P_k v \|_{L^2 ( \R^d) }  \chi_{(\abs{n- k - m} \leq 3)}, \\
	& \les \| P_m u \|_{L^2 (\R^d)} \| P_k v \|_{L^2 ( \R^d) }  \chi_{(\abs{n- k - m} \leq 3)}.
\end{align*}

\noi Hence,
\begin{align*}
\| uv \|_{M^{2,1}_0 (\R^d)} & \leq \sum_{n \in \Z^d} \sum_{m,k \in \Z^d} \|  P_n (P_m u P_k v ) \|_{L^2 ( \R^d)}, \\
	& \les  \sum_{n \in \Z^d} \sum_{m,k \in \Z^d} \| P_m u \|_{L^2 (\R^d)} \| P_k v \|_{L^2 ( \R^d) }  \chi_{(\abs{n- k - m} \leq 3)}, \\
	& \les  \sum_{m,k \in \Z^d} \| P_m u \|_{L^2 (\R^d)} \| P_k v \|_{L^2 ( \R^d) }
\end{align*}

\noi which ends the proof.

\end{proof}

\subsection{Proofs of norm inflation in other spaces}

In this subsection, we claim that the following result is true:

\begin{theorem}
\label{THM:NIwithILORforFLMW}
Assume that $s < 0$ and $1 \leq q \leq \infty$ and let
\begin{equation}
\label{EQ:DefZsq}
Z_s^q \coloneqq \FL^{s,q}(\M) \quad \text{ or } \quad M_s^{2,q}(\R) \quad \text{ or } \quad W_s^{2,q}(\R).
\end{equation}

\noi Then, norm inflation with infinite loss of regularity occurs at the origin for \eqref{imBq} in $Z_s^q$.

\noi Namely, let $\theta \in \R$, $s < 0$ and fix $\vec{u}_0 \in Z_s^q \times Z_s^q$. Given any $\eps > 0$, there exists a solution $u_\eps \in C([0,T], Z_s^q)$ to \eqref{imBq} with $T > 0$ and $t_\eps \in (0, \eps)$ such that
$$
\| (u_\eps (0), \dt u_\eps (0) ) - \vec{u}_0 \|_{Z_s^q \times Z_s^q} < \eps \qquad \text{ and } \qquad \| u_\eps (t_\eps ) \|_{Z_\theta^q} > \eps^{-1}.
$$

\end{theorem} 

As explained before, the idea to prove our result is to use the same construction as in Section \ref{Sec3}. Actually, we show that, by keeping the same notations and choosing $A$ to be constant, we have estimates that are equivalent to the ones given in Lemma \ref{LEM:MultilinearEst} and Proposition \ref{PROP:EstSecondPicard}. Then, the proof of Theorem \ref{THM:NIwithILORforFLMW} follows with the same argument.

\subsubsection{The case of Fourier-Lebesgue spaces}

In the case of Fourier-Lebesgue spaces $\FL^{s,q} (\M)$, we get from the same computations as for the Sobolev spaces $H^s (\M)$:

\begin{proposition}
\label{PROP:EstforFL}
Let $s < 0$ and $1 \leq q \leq \infty$. We have then 
\begin{equation}
\label{EQ:normXi0FL}
\| \Xi_0 [\vec{u}_{0, n}](T) \|_{\FL^{s,q}} \les 1 + R N^s A^{\frac 1q},
\end{equation}

\noi as well as
\begin{equation}
\label{EQ:EstXi1UpperFL}
\|\Xi_1 [\vec{u}_{0, n}] (T) - \Xi_1 [\vec{\phi}_{ n}] (T)  \|_{\FL^{s,q}} \les T^{2} (RA)^{(k-1)} \| \vec{u}_0 \|_{\FLv^{q}( \M)},
\end{equation}

\noi and, for any $j \geq 2$, 
\begin{equation}
\label{EQ:EstXijFL}
\|\Xi_j [\vec{u}_{0, n}] (T) \|_{\FL^{s,q}} \les T^{2j} (RA)^{(k-1)j} \big( \| \vec{u}_0 \|_{\FLv^{q}( \M)} + Rf_{s,q} (A) \big)
\end{equation}

\noi where
\begin{equation}
\label{EQ:deffsqFL}
f_{s,q}(A) = \| \jb{\xi}^s \|_{L^q_\xi (Q_A) }.
\end{equation}

\noi Also, if $A \ll N$, we have the following lower bound:
\begin{equation}
\label{EQ:EstSecondPicardFL}
\| \Xi_1 (\vec{\phi}_n) (T) \|_{ \FL^{s,q}} \ges R^k T^2 A^{k - 1 + \frac 1q + s}
\end{equation}
\end{proposition}

\noi Under the condition $A = 10$, we observe that \eqref{EQ:normXi0FL}, \eqref{EQ:EstXi1UpperFL}, \eqref{EQ:EstXijFL} and \eqref{EQ:EstSecondPicardFL} are essentially equivalent respectively to \eqref{EQ:MultilinearEst2}, \eqref{EQ:MultilinearEst3}, \eqref{EQ:MultilinearEst4} and \eqref{EQ:EstSecondPicard}, so that the expected result follows from the same argument. There is one condition to change though, which is $\| \vec{u}_0 \|_{\FLv^{q}( \M)}  \les 1$ instead of $\| \vec{u}_0 \|_{\H^0 ( \M)} \les 1$, but this is quite reasonable since our initial data $\vec{u}_0$ is fixed.

\subsubsection{The case of modulation spaces}

The computations in this section follow the same ideas as in the previous section, but the modulation spaces make them a bit trickier to apply. We first state our results:

\begin{proposition}
\label{PROP:EstForMs2qspaces}
Let $s < 0$ and $1 \leq q \leq \infty$. Keeping the same notations as before, and asssuming

$$
\| \vec{u}_0 \|_{M^{2,q}_s \times M^{2,q}_s} \les 1 \quad \text{ and }  \quad \| \vec{u}_0 \|_{M^{2,1}_0 \times M^{2,1}_0}  \ll  R \| (1 + \abs{n})^s \|_{\ell^q ( 1 \leq \abs{n} \leq A)},
$$

\noi we have 
\begin{equation}
\label{EQ:normXi0M}
\| \Xi_0 [\vec{u}_{0,n}](T) \|_{M_s^{2,q}} \les 1 + R N^s A^{\frac 12},
\end{equation}

\noi as well as
\begin{equation}
\label{EQ:EstXi1UpperM}
\|\Xi_1 [\vec{u}_{0, n}] (T) - \Xi_1 [\vec{\phi}_{ n}] (T)  \|_{M_s^{2,q}} \les T^{2} (RA)^{(k-1)} \| \vec{u}_0 \|_{M^{2,1}_0( \M)},
\end{equation}

\noi and, for any $j \geq 2$
\begin{equation}
\label{EQ:EstXijFL}
\| \Xi_j [\vec{u}_{0,n}] (T) \|_{M_s^{2,q}}   \les T^{2j} (RA)^{(k-1)j} R \| (1 + \abs{n})^s \|_{\ell^q ( 1 \leq \abs{n} \leq A)}.
\end{equation}

\noi Also, if $A \ll N$, we have the following upper bound:
\begin{equation}
\label{EQ:EstSecondPicardFL}
\| \Xi_1 (\vec{\phi}_n) (T) \|_{M_s^{2,q}} \ges R^k T^2 A^{k - 1}A^{ \frac 1q + s}.
\end{equation}
\end{proposition}

Again, we see that if we choose $A$ to be constant, all these estimates are equivalent to the estimates we had for the Sobolev spaces, and the norm inflation result becomes straightforward. We include the proof of Proposition \ref{PROP:EstForMs2qspaces} both for completeness and to put an emphasis on the differences between these spaces and Sobolev spaces.

\begin{proof}[Proof of Proposition \ref{PROP:EstForMs2qspaces}]

First, we have from $\abs{ \cos x } \leq 1$ and $\abs{\sin x} \les \abs{x}$,
\begin{equation}
\label{EQ:EstLinearM}
\| S(t)(u_0 , u_1 ) \|_{M_s^{2,q}} \les \| u_0 \|_{M_s^{2,q}} + \abs{t} \| u_1 \|_{M_s^{2,q}} \les \| (u_0, u_1) \|_{M^{2,q}_s \times M^{2,q}_s}
\end{equation}

\noi for any $0 \leq t \leq 1$ and $u_0, u_1 \in M^{2,q}_s$. Besides, we also have 
\begin{align*}
\| \phi_n \|_{M_s^{2,q}}  & = R \big\| (1 + \abs{n})^s \| P_n \F^{-1}(\sum_{\eta \in \Si} \chi_{\eta + Q_A}) \|_{L^2} \big\|_{\ell^q_n} \\
	& =  R \big\| (1 + \abs{n})^s \| \s_n \sum_{\eta \in \Si} \chi_{\eta + Q_A} \|_{L^2} \big\|_{\ell^q_n}
\end{align*}

\noi but since the sets $\eta + Q_A$ are disjoint for any two $\eta_1 \neq \eta_2$ in $\Si$, we get
$$
\left\| \s_n \sum_{\eta \in \Si} \chi_{\eta + Q_A} \right\|^2_{L^2} \sim \chi_{N + Q_A} (n) \int_{\R} \s^2_n (\xi) \chi^2_{n + Q_A} (\xi) \mathrm{d}\xi \sim \chi_{N + Q_A} (n)
$$

\noi and
$$
\| \phi_n \|_{M_s^{2,q}} \sim R \| (1 + \abs{n} )^s \chi_{N + Q_A} (n)\|_{\ell^q_n} \sim RN^s A^{\frac 1q}
$$

\noi which, combined with \eqref{EQ:EstLinearM}, proves our first estimate. Besides, \eqref{EQ:EmbeddingsM} and the algebra property of $M^{2,1}_0$ yield, in a similar argument as for \eqref{EQ:MultilinearEst3}:
\begin{align*}
\|\Xi_1 [\vec{u}_{0, n}] (T) - \Xi_1 [\vec{\phi}_{ n}] (T)  \|_{M_s^{2,q}} & \les \|\Xi_1 [\vec{u}_{0, n}] (T) - \Xi_1 [\vec{\phi}_{ n}] (T)  \|_{M_0^{2,1}}, \\
	& \les T^2 \| \vec{u}_0 \|_{M^{2,1}_0 \times M^{2,1}_0} \big( \| \vec{u}_0 \|_{M^{2,1}_0 \times M^{2,1}_0} + \| \vec{\phi}_n \|_{M^{2,1}_0 \times M^{2,1}_0} \big)^{k-1}, \\
	& \les T^2 \| (u_0, u_1) \|_{M^{2,1}_0 \times M^{2,1}_0} (RA^{\frac 12})^{k-1}.
\end{align*}

\noi Hence \eqref{EQ:EstXi1UpperM}.

Now, let $j \geq 2$. Using the same argument on the support as in the proof for \eqref{EQ:MultilinearEst4}, we have 
\begin{align*}
\| \s_n \F[\Xi_j (\vec{\phi}_n)](T) \|_{L^2_\xi} & \leq \| \s_n \|_{L^2_\xi (\supp \F[ \Xi_j  (\vec{\phi}_n)](T))} \| \Xi_j (\vec{\phi}_n) (T) \|_{\FL^\infty} \\
	& \les \| \s_n\|_{L^2_\xi (Q_A \cap Q_n)} C^j T^{2j} \| \vec{\phi}_n \|^{(k-1)j-1}_{\FLv^1} \| \vec{\phi}_n \|^2_{\H^0} \\
	& \les \mathds{1}_{\{ n \in Q_A\}} T^{2j}(RA)^{(k-1)j} R
\end{align*}

\noi so that
$$
\| \Xi_j [\vec{\phi}_n] (T) \|_{M_s^{2,q}}  \les T^{2j} (RA)^{(k-1)j} R \| (1 + \abs{n})^s \mathds{1}_{\{ n \in Q_A \}} \|_{\ell^q_n}.
$$

\noi Besides, a similar argument as for \eqref{EQ:EstXi1UpperM} yields
\begin{align*}
\|\Xi_j [\vec{u}_{0, n}] (T) - \Xi_j [\vec{\phi}_{ n}] (T)  \|_{M_s^{2,q}} & \les \|\Xi_j [\vec{u}_{0, n}] (T) - \Xi_j [\vec{\phi}_{ n}] (T)  \|_{M_0^{2,1}}, \\
	& \les T^{2j} \| \vec{u}_0 \|_{M^{2,1}_0 \times M^{2,1}_0} \big( \| \vec{u}_0 \|_{M^{2,1}_0 \times M^{2,1}_0} + \| \vec{\phi}_n \|_{M^{2,1}_0 \times M^{2,1}_0} \big)^{(k-1)j}, \\
	& \les T^{2j} \| \vec{u}_0 \|_{M^{2,1}_0 \times M^{2,1}_0} (RA^{\frac 12})^{(k-1)j},
\end{align*}

\noi which gives our third estimate.

Finally, our fourth estimate follows from a straightforward adaptaton of the proof of Proposition \ref{PROP:EstSecondPicard}, the only difference being the norm used.

\end{proof}

\subsubsection{The case of Wiener amalgam spaces}

This section relies a lot on the interactions between modulation and Wiener amalgam spaces.  Indeed, using \eqref{EQ:RelationsMandW_end} and Proposition \ref{PROP:EstForMs2qspaces}, we get respectively
\begin{equation}
\label{EQ:EstForWs2q_Xi0}
\| \Xi_0 [\vec{u}_{0,n}](T) \|_{W_s^{2,q}} \les \| \Xi_0 [\vec{u}_{0,n}](T) \|_{M_s^{2,\min(2,q)}} \les 1 + R N^s A^{\frac 12},
\end{equation}

\noi as well as
\begin{align}
\label{EQ:EstForWs2q_Xi1Upper}
\| \Xi_1 [\vec{u}_{0,n}] (T) - \Xi_1 [\vec {\phi}_n](T) \|_{W_s^{2,q}}  & \les \| \Xi_1 [\vec{u}_{0,n}] (T) - \Xi_1 [\vec {\phi}_n](T) \|_{M_s^{2,\min(2,q)}}\\
	& \les T^{2} (RA)^{(k-1)} \| \vec{u}_0 \|_{M^{2,1}_0( \M)},
\end{align}

\noi and, for any $j \geq 2$, 
\begin{align}
\label{EQ:EstForWs2q_Xij}
\| \Xi_j [\vec{u}_{0,n}] (T) \|_{W_s^{2,q}}  & \les \| \Xi_j [\vec{u}_{0,n}] (T) \|_{M_s^{2,\min(2,q)}}\\
	& \les T^{2j} (RA)^{(k-1)j} R \| (1 + \abs{n})^s \|_{\ell^{\min(2,q)} ( 1 \leq \abs{n} \leq A)}.
\end{align}

\noi Besides, we also get the following lower bound:
\begin{equation}
\label{EQ:EstForWs2q_Xi1}
\| \Xi_1 (\vec{\phi}_n) (T) \|_{W_s^{2,q}} \ges \| \Xi_1 (\vec{\phi}_n) (T) \|_{W_s^{2,\max(2,q)}} \ges R^k T^2 A^{k - 1}A^{ \frac 1q + s}.
\end{equation}

\noi Again, these estimates are equivalent to the ones we had for Sobolev spaces under the condition $A = 10$, so norm inflation with infinite loss of regularity follows.

\section*{Acknowledgments} The author would like to thank Tadahiro Oh for suggesting this problem and his continuous support throughout this work. The author acknowledges support from Tadahiro Oh's ERC grant (no. 864138 ``SingStochDispDyn"). The author would also like to thank Younes Zine for several helpful discussions.

\end{document}